\documentclass[12pt]{amsart}
\usepackage{version,amsmath, amssymb}
\usepackage{color}
\usepackage{amsthm, amsopn, amsfonts, xypic}
\usepackage[colorlinks=true]{hyperref}
\usepackage{mathrsfs}
\usepackage[utf8]{inputenc}
\usepackage{tcolorbox}
\hypersetup{urlcolor=blue, citecolor=blue}
\calclayout

\newtheorem{theorem}{Theorem}[section]
\newtheorem{lemma}[theorem]{Lemma}

\theoremstyle{definition}
\newtheorem{definition}[theorem]{Definition}

\theoremstyle{remark}
\newtheorem{remark}[theorem]{Remark}

\numberwithin{equation}{section}

\newcommand{\R}{{\mathbb{R}}}

\newcommand{\pa}{\partial}

\newcommand{\grad}{\nabla_x}
\newcommand{\flux}{\text{flux}}
\newcommand{\gab}{g^{\alpha\beta}}
\newcommand{\hab}{h^{\alpha\beta}}

\newcommand{\pab}{\partial_\beta}
\newcommand{\paa}{\partial_\alpha}

\newcommand{\pat}{\partial_t}
\newcommand{\pai}{\partial_i}
\newcommand{\paj}{\partial_j}

\newcommand{\pao}{\pa_\omega}

\newcommand{\nua}{\nu_\alpha}

\newcommand{\eps}{\epsilon}

\newcommand{\la}{\langle}
\newcommand{\ra}{\rangle}
\newcommand{\im}{\la x \ra}

\newcommand{\ls}{\lesssim}
\newcommand{\gs}{\gtrsim}
\newcommand{\bd}{\text{BD}_h}
\newcommand{\gext}{\Gamma^{\text{ext}}}
\newcommand{\hll}{h^{\underline L\underline L}}
\newcommand{\nd}{\underline L}
\newcommand{\ed}{|\nabla u|^2}
\newcommand{\LE}{\|u\|^2_{LE^1[T_1,T_2]}}

\begin{document}

\title{Scattering for critical wave equations with variable coefficients}

\author{Shi-Zhuo Looi}
\address{Department of Mathematics, University of Kentucky, Lexington, 
  KY  40506}
\email{shizhuo.looi@uky.edu}
\author{Mihai Tohaneanu}
\address{Department of Mathematics, University of Kentucky, Lexington, 
  KY  40506}
\email{mihai.tohaneanu@uky.edu}

\begin{abstract}

We prove that solutions to the quintic semilinear wave equation with variable coefficients in $\R^{1+3}$ scatter to a solution to the corresponding linear wave equation. The coefficients are small and decay as $|x|\to\infty$, but are allowed to be time dependent. The proof uses local energy decay estimates to establish the decay of the $L^6$ norm of the solution as $t\to\infty$. 
\end{abstract}

\maketitle

\vspace{-0.35in}

\tableofcontents

\section{Introduction}

In Minkowski space, solutions of the equation $\Box u = |u|^{p-1} u$ with $\Box=-\pat^2 + \Delta$ have a conserved and positive-definite energy $$E(t) = \int_{\R^3} \frac{1}{2}|\nabla u(t,x)|^2 + \frac{1}{p+1}|u(t,x)|^{p+1} \,dx$$ and the scaling symmetry $$u(t,x)\mapsto \lambda^{\frac{2}{1-p}} u(\frac{t}{\lambda}, \frac{x}{\lambda}).$$ In three dimensions, the exponent $p=5$ is called the energy-critical exponent, because solutions of the equation have an energy that is invariant under the scaling symmetry. 

For the Cauchy problem with initial data in the energy space $\dot H^1\times L^2$, local well-posedness is proven for $1<p\leq 5$ by Strichartz estimates. Global existence for small initial data is a straightforward adaptation of the proof of local existence. In addition, there is global existence for large initial data due to the existence of a blowup criterion, which informally says that the energy cannot concentrate at any point in spacetime. Moreover, given any finite-energy initial data there is a unique global solution with finite energy lying in $L^4L^{12}([0,\infty)\times\R^3)$; these solutions are known as strong (or Shatah-Struwe) solutions. See \cite{gsv},\cite{g1},\cite{g2},\cite{g3},\cite{k},\cite{n1},\cite{n2},\cite{p},\cite{r},\cite{ss1},\cite{ss2},\cite{s} for details and more. The results in \cite{bs} and \cite{bg} then combine to prove scattering of solutions with finite-energy initial data using a profile decomposition, which describes the failure of a sequence of uniformly bounded solutions to the free wave equation to be compact in the sense of Strichartz estimates. 

This paper considers the equation
\begin{equation}\label{main}
\begin{cases}
Pu(t,x) = u(t,x)^5 & (t,x) \in (0,\infty) \times \R^3, \quad P = \paa\gab\pab \\
u[0] \in \dot H^1 \times L^2
\end{cases}
\end{equation}

Global existence and uniqueness of strong solutions (lying in $C(\R_t, \dot H^1) \cap L_{loc}^5L^{10}$) was shown in \cite{im} in the stationary setting. A similar result for classical solutions in the non-stationary setting was shown in \cite{zl}. These results require minimal assumptions on the coefficients, as eliminating the blowup scenario only requires local-in-time arguments.

Our main theorem establishes scattering of strong solutions to \eqref{main} for certain small, asymptotically flat perturbations of the Minkowski metric.

\begin{theorem} \label{1.1}
Let $\gab(t,x)$ be a Lorentzian metric, let $P = \paa\gab\pab$, and let $h = g - m$ denote the perturbative terms of the metric $g$. The unique global strong solution to the Cauchy problem \eqref{main} scatters in the energy space $\dot H^1\times L^2$ provided that
\begin{equation}\label{hest}
|h|\ls \eps \frac{\la t-r \ra^{1/2}}{\im^\gamma \la t+r\ra^{1/2}},
\end{equation}
 \begin{equation}\label{hbadest}
 |\hll| \lesssim \eps\frac{\la t-r \ra}{\im^\gamma \la t+r\ra},
 \end{equation}
\begin{equation}\label{derhest}
|\pa^J h| \ls \eps \frac{1}{\im^{|J|+\gamma}} \ \ \text{for } |J| = 1 \text{ and } |J| = 2
\end{equation}
 where $\gamma>0$ is an arbitrarily small constant and $\epsilon>0$ is a sufficiently small constant. In these assumptions, $\pa^J h$ denotes $\pa^J \hab$ for all multi-indices $\alpha$ and $\beta$, and $\hll=h^{\alpha\beta}\nd_\alpha\nd_\beta$, where $\nd = \sum_{i=1}^3\frac{x^i}{|x|}\pa_{x_i} - \pat$ and we lower indices with respect to the Minkowski metric.
\end{theorem}
		
		This says that the unique global solution of the non-linear problem on small, asymptotically flat perturbations of Minkowski space that have appropriate decay at infinity behave, in the asymptotic sense, like the solution to the linear homogeneous problem $Pu =0$, at least in the energy space. 
		
\begin{remark}
	Theorem~\ref{1.1} also holds if we replace $P$ by the geometric wave operator 
\[
\Box_g=\frac{1}{\sqrt{|g|}} \paa \sqrt{|g|} \gab \pab, \quad |g|:=|\det \gab|.
\] 

Indeed, in the estimates below one integrates with respect to the volume form $\sqrt|g| dt dx$, and uses the fact that $\sqrt{|g|}\gab\approx\gab$.  There are extra error terms of the form $\pa \sqrt{|g|} u^6$ arising, which can be absorbed since by \eqref{derhest} we have
\[
\pa \sqrt{|g|} \lesssim \frac{\epsilon}{\la x\ra^{1+\gamma}}
\]
	
\end{remark}

\begin{remark} 
A key tool in proving scattering on variable-coefficient backgrounds is local energy decay. Such an estimate was proven in \cite{m}, \cite{sr}, and \cite{kss} for Minkowski space and in \cite{ms}, \cite{mt1} for perturbations of Minkowski space, and became a valuable tool in the study of both linear and nonlinear problems. In particular, they imply Strichartz estimates on certain variable coefficients backgrounds, see  \cite{mt2}. Our result is one of several showing that local energy decay is fruitful for understanding the long-time behavior and asymptotics of solutions to nonlinear dispersive equations on variable coefficients backgrounds.
\end{remark}

\begin{remark}
For the energy-critical problem on Minkowski space, global a priori estimates were proven in \cite{n2}, from which scattering for the wave and Klein-Gordon equations was deduced. Profile decompositions akin to \cite{bg} have been shown for waves on hyperbolic space in \cite{los}. Analogous results in the exterior of obstacles were obtained in \cite{kss}, \cite{smithsogge}, \cite{blp}. Scattering on Riemannian manifolds for a class of non-trapping obstacles close to the two-convex framework has been shown as well for the energy-critical non-linear problem (\cite{l}).

For the nonlinear Schr\"odinger equation, the energy-critical problem for the defocusing quintic problem with initial data in the energy space for small, compactly supported perturbations of the Euclidean metric also exhibits scattering to linear solutions for all finite-energy data, as shown in \cite{j}. In the exterior of a strictly convex obstacle, global well-posedness and scattering for all initial data in the energy space for the NLS was shown in \cite{kvz}. 
\end{remark}

\section{Notation and Preliminaries} \label{notation}
We fix the spatial dimension to be $d=3$ and define $\Box = - \pa_t^2+\Delta$ and $P = \paa \gab \pab$ where $g = g(t,x)$ is a Lorentzian 
metric; we also write $m$ for the Minkowski metric and $h = g-m$. We write either $X \lesssim Y$ or $X=O(Y)$ to indicate that $$|X| \leq CY$$ (rather than $X\leq CY$) for some absolute constant $C$ which may vary by line. Similarly, $X\approx Y$ means that there are constants $0 < C_1 < C_2$ so that
\[
C_1 |X| \leq |Y| \leq C_2 |X|.
\]
We let $\la r \ra = \la x \ra = (1 + |x|^2)^{1/2}$. We write $\nabla = (\pat,\nabla_x)$ for the spacetime gradient. Throughout the paper, we use the Einstein summation convention, and we let Greek (resp. Latin) indices denote spacetime (resp. space) indices. We write $u[T] = (u(T,x), \pat u(T,x))$.

The energy of the solution $u$ is defined to be 
	$$E(t) = \int_{\R^3} \frac{1}{2}|\nabla u(t,x)|^2 + \frac16 u^6\,dx.$$

We will also use the notation
$$E_K(t) := \int_{ K} \frac{1}2|\nabla u(t,x)|^2 + \frac16 u^6 \,dx$$	
for some subset $K$ of $\R^3$.

For any $(f, g)\in \dot H^1 \times L^2$, we denote by $S(t,s)(f, g)$ the unique solution $u\in C(\R_t, \dot H^1)$ with $\pa_t u \in C(\R_t, L^2)$ to the equation
\begin{equation}\label{linear}
\begin{cases}
P u = 0  & (t,x) \in (s,\infty) \times \R^3 \\
u[s] = (f, g)
\end{cases}\end{equation}

Let
\[
X=\{C(\R_t, \dot H^1)\cap L_{loc}^5L^{10}\}\times C(\R_t, L^2)
\]
and for any closed, finite interval 
\[
X(I) = \{C(I, \dot H^1)\cap L^5(I)L^{10}\}\times C(I, L^2)
\]

Consider the Cauchy problem 
\begin{equation} \label{cp3}
\begin{cases}
Pu(t,x) = u(t,x)^5 & (t,x) \in (0,\infty) \times \R^3 \\
u[0] \in \dot H^1 \times L^2
\end{cases}
\end{equation}

By Duhamel's formula, classical solutions to \eqref{cp3} satisfy
\begin{equation}\label{strong}
u(t) = S(t, 0)u[0] + \int_0^t \frac{1}{g^{00}}S(t,s)(0, u^5(s)) ds 
\end{equation}

We can thus define a strong solution to be a solution of \eqref{strong} so that $(u, \partial_t u)$ also lies in $X$.

The results of \cite{im} and \cite{zl} show that, for smooth initial data, there is a unique global classical solution to \eqref{cp3} that is also a strong solution. Moreover, this result is extended to initial data in the energy space in \cite{im} for time-independent coefficients, and the same argument can be used to prove it in the time-dependent case. We will be interested in studying the asymptotic properties of the unique strong solution in the energy space, in particular the fact that it approaches a solution to the linear equation in the energy space.

\begin{definition}[Scattering in the energy space]
We say that the solution $u$ to \eqref{cp3} \textit{scatters in the energy space} if there exists $(f,g) \in \dot H^1 \times L^2$ such that 
$$ \lim_{t\to\infty} \| u(t) - S(t,0)(f,g) \|_{\dot H^1 \times L^2} = 0.$$
\end{definition}

\section{Energy conservation and local energy decay in Minkowski space}

In order to motivate the discussion that follows, we devote this section toward certain key results for the linear problem in the setting of Minkowski space. Consider the Cauchy problem
\begin{equation} \begin{cases} \label{cp1}
\Box u = 0 & (t,x) \in (0,\infty) \times \R^3 \\
u[0] \in \dot H^1 \times L^2
\end{cases} \end{equation}

The energy of the solution $u$ to \eqref{cp1} is defined to be 
	$$E^{lin}(t) = \int_{\R^3} \frac{1}{2}|\nabla u(t,x)|^2\,dx$$
and it is conserved: for all $T \geq 0$, $E^{lin}(T)=E^{lin}(0) = E$. 

	The solution $u$ to \eqref{cp1} also satisfies the local energy estimate
	$$\iint_{[0,T] \times \R^3} \frac{|\nabla u|^2}{\langle r\rangle^{1+\gamma}} + \frac{u^2}{\langle r\rangle^{3+\gamma}} \,dxdt \lesssim E$$ where $\gamma>0$ is an arbitrarily small fixed constant.
		This estimate is proven by multiplying both sides of the equation by $a(r)u + b(r)\pa_r u + C\pat u$ in the region $[0,T] \times \R^3$, with: $$b(r) = \sum_{j=0}^\infty 2^{-j\gamma}\frac{r}{r+2^j},$$ where for each $j$, this is a function catered to the region $r \approx 2^j$ with the factor $2^{-j\gamma}$, where the small number $\gamma>0$ is introduced in order to obtain convergence of the series; $a(r) = b(r)/r$; and $C>0$ is a constant chosen to be sufficiently large. 

More precisely, one obtains 
	\begin{align*}
	\iint_{[0,T]\times\R^3} \Box u b(r)\pa_r u \,dxdt &= - \iint b'(\pa_ru)^2 + \frac{b}{r} |\pa_\omega u|^2 + \frac{1}{2}(b' + 2\frac{b}{r})(u_t^2 - |\nabla_x u|^2) \,dxdt \\
			&\ \ + \int_{\R^3} - f \pa_r u \pat u |_0^T \,dx,
	\end{align*} where $|\pao u|^2 := |\grad u|^2 - |\pa_r u|^2$, and 
	
	\begin{align*}
	\iint_{[0,T]\times\R^3} \Box u a(r) u \,dxdt = \iint a(u_t^2-|\grad u|^2) + \frac{1}{2}\Delta au^2 \,dxdt + \frac{1}{2}\int_{\R^3} -au\pat u|_0^T\,dx.		
	\end{align*} 
	
	Since $a = O(\la r \ra^{-1})$ and $b = O(1)$, Hardy's inequality $\int_{\R^3} u^2/r^2 \,dx \lesssim \int_{\R^3} | \grad u|^2\,dx$ shows that there exists a sufficiently large constant $C>0$ such that 
	$$\int_{\{t\}\times\R^3}  b\pat u \pa_r u + a\pat u u + C|\nabla u|^2\,dx \approx E^{lin}(t)$$
for all $t \geq 0$. We obtain 
\[
	E^{lin}(T) + \iint_{[0,T]\times\R^3} \frac{1}{2}b'(u_r^2 + u_t^2) - ( \frac{1}{2}b' - \frac{b}{r} ) | \pao u|^2 - \frac{\Delta a}{2} u^2\,dxdt \lesssim E^{lin}(0)
\] 

One can check directly that
\[
b' \gtrsim r^{-1-\gamma}, \quad b/r - \frac{1}{2}b'  \gs r^{-1-\gamma}, \quad -\Delta a \gs r^{-3-\gamma}
\]
and thus
\begin{equation}\label{LEest}
\frac{1}{2}b'(u_r^2 + u_t^2) + (- \frac{1}{2}b' + \frac{b}{r} ) | \pao u|^2 - \frac{\Delta a}{2} u^2 \gtrsim \frac{|\nabla u|^2}{\langle r\rangle^{1+\gamma}} + \frac{u^2}{\langle r\rangle^{3+\gamma}}
\end{equation}

which finishes the proof.
	
\section{Uniform energy bounds and local energy decay for the nonlinear problem on perturbations}

We now come to certain key tools, analogous to the results presented in the previous section, that will be used in the proof of scattering for the non-linear problem on certain perturbations of Minkowski space that have appropriate decay at infinity.

If $g(t,x)$ is a non-stationary metric that satisfies certain decay conditions, and $u$ is a solution of the Cauchy problem \eqref{cp3}, then we have uniform energy bounds: with the energy now defined to be
	$$E(t) := \int_{ \{t\} \times \R^3 } \frac{1}{2}|\nabla u|^2 + \frac{1}{6}|u|^6\,dx,$$
	if $g$ and its derivatives satisfy certain decay conditions, then $$E(T) \lesssim E(0) := E$$ for some implicit constant that is independent of $T$. In fact, we may prove local energy decay and uniform energy bounds in one fell swoop for \eqref{cp3}, as the following proposition shows.

\begin{theorem}[Integrated local energy decay for the nonlinear Cauchy problem]
Let $u$ be a solution of \eqref{cp3} and let $|J| \leq 1$ be a multi-index. If $\pa^J \hab \lesssim \epsilon \la r \ra^{-|J|-\gamma}$ where $\gamma>0$ is an arbitrarily small constant and $\epsilon>0$ is a sufficiently small constant then \begin{equation}\label{LE}
\|u\|_{LE^1[T_1,T_2]}^2 + E(T_2) \lesssim E(T_1)
\end{equation} for some implicit constant that is independent of $T_1$ and $T_2$, where
\[
\|u\|_{LE^1[T_1,T_2]}^2 = \iint_{[T_1,T_2] \times \R^3} \frac{|\nabla u|^2}{\la r\ra^{1+\gamma}}+ \frac{u^2}{\la r\ra^{3+\gamma}} + \frac{|u|^6}{\la r\ra}\,dxdt
\]
\end{theorem}

Let us first assume that $u$ is a classical solution to the equation
\[
Pu = u^5 + F
\]

Following \cite{mt1} and the discussion in the previous section, we multiply the equation by $a(r)u+b(r)\pa_r u + C\pa_t u$, with 
\[b(r) = \sum_{j=0}^\infty 2^{-j\gamma} \frac{r}{r + 2^j}, \quad a(r) = b(r)/r. \] 
Upon integrating by parts, we obtain 

	\begin{align*} E(T_2) + &\iint_{[T_1,T_2] \times \R^3} \frac{1}{2}b'(u_r^2 + u_t^2) - ( \frac{1}{2}b' - \frac{b}{r} ) | \pao u|^2 - \frac{\Delta a}{2} u^2 + Err\\
	& \ \ + (\frac23 a(r) - \frac16 b'(r))u^6 \,dxdt \lesssim E(T_1) + \iint_{[T_1,T_2] \times \R^3} |F|(|\pa u| + \frac{|u|}{\la r\ra}) \,dxdt.\end{align*} 
where the error satisfies
\[	
Err \lesssim (\frac{|h|}{\la x \ra} + |\nabla h|)(| \nabla u|^2 + |\nabla u|\frac{|u|}{\la x \ra})	
\]
	Since $\pa^J \hab \lesssim \epsilon \la r \ra^{-|J|-\gamma}$, we can estimate by Cauchy-Schwarz
\[
Err \lesssim \epsilon \Bigl(\frac{|\nabla u|^2}{\la r\ra^{1+\gamma}}+ \frac{u^2}{\la r\ra^{3+\gamma}}\Bigr) 
\]	

 Moreover,
\[
\frac23 a(r) - \frac16 b'(r) = \sum_{j=0}^\infty 2^{-j\gamma} \left(\frac{1}{2} \frac{1}{r + 2^j} +\frac{1}{6} \frac{r}{(r + 2^j)^2} \right) \gtrsim \frac1r\
\] 

Taking \eqref{LEest} into account, and applying H\"older and Hardy to control the inhomogeneity, we get
\begin{equation}\label{LEcl}
\|u\|_{LE^1[T_1,T_2]}^2 + E(T_2) \lesssim E(T_1) + \|F\|_{L^1[T_1, T_2]L^2}\|\nabla u\|_{L^{\infty}[T_1, T_2]L^2}
\end{equation}
	
Consider now a strong solution $u$, and a sequence of classical solutions $u_n$ so that $u_n(T_1)\to u(T_1)$ in the energy norm. After dividing the interval $I=[T_1, T_2]$ into finitely many intervals so that the $L^5L^{10}$ norm of $u$ is suitably small on each interval, a contraction argument shows that $u_n \to u$ in $X(I)$. In particular this implies that $u_n^5$ is a Cauchy sequence in $L^1[T_1, T_2]L^2$, and thus by \eqref{LEcl} we must have $u_n \to u$ in $LE^1[T_1,T_2]$. The desired conclusion \eqref{LE} now follows.

\section{$L^6$ norm decay of solutions in Minkowski space}

	In order to motivate the the next section, which contains the main result and its proof, in this section we shall present the highlights of the proof of $L^6$ norm decay in Minkowski space for the non-linear problem, as done in \cite{bs}. 
	
\subsection{More notation}
 First, we fix some notation which will be used for the rest of the paper. Let 
	$$\Gamma = \{(t,x) : |x| - c < t, t>0 \}$$
be a forward solid light cone with $c\geq 0$ to be determined and let $\Gamma(I) = \Gamma \cap (I \times \R^3)$ where $I \subset [0,\infty)$ is a time interval. Let $D(T) = \{ (t,x) : t = T, \, |x| - c < t \}$ denote its $t=T$ slices and let $L(I) = \{ (t,x) : t \in I, |x| - c = t \}$ denote the lateral boundary of $\Gamma(I)$ with $$L_c(I):= L(I)$$ also used for emphasis, but usually we shall simply write $L(I)$.

\ Consider next the Cauchy problem 
\begin{equation} \label{cp4}
\begin{cases}
\Box u = u^5 & (t,x) \in (0,\infty) \times \R^3 \\
u[0] \in \dot H^1 \times L^2
\end{cases}
\end{equation}

	We now sketch a proof of the $L^6$ norm decay in the energy space for solutions to \eqref{cp4} (see \cite{ss2}, \cite{bg} for more details). We will adapt this proof to the variable coefficient case in the next section.
	
	Given $\delta>0$, pick $c$ sufficiently large so that the energy in the exterior region $|x| > c$ is less than $\delta/2$. Now given any $c \geq 0$, the flux on the time interval $I$ is defined to be the integral on the lateral boundary arising from multiplying the equation $\Box u = u^5$ by $\pat u$, namely 
	$$\flux(I) = \int_{L(I)} \frac{1}{2}|\pao u|^2 + \frac{1}{2}|\pat u + \pa_r u|^2 + \frac{1}{6}u^6\,\frac{d\sigma}{\sqrt{2}}$$

It is clear that the flux is non-negative. As the upper and lower limits of $I$ approach infinity, the flux decays, as an application of the divergence theorem in the interior of the $\Gamma(I)$ region shows. More precisely, if $I=[T_1,T_2]$, one obtains $$E_{|x|<c+T_2}(T_2) = E_{|x|<c+T_1}(T_1) + \flux(I).$$

Thus $E_{|x|<c+t}(t)$ is monotone non-decreasing; moreover, it is bounded; therefore it converges to a limit as $t \to \infty$, as claimed.

We now multiply the \eqref{cp4} by $Xu := (t+c) \pat u + x^i \pa_i u + u$ and apply the divergence theorem in $\Gamma(I)$. We obtain
\begin{equation}
P(T_2) + \iint_{\Gamma(I)} \frac{u^6}{3} = P(T_1) + \int_{L(I)} (t+c) \left( \frac{Xu}{t+c}\right)^2 \,\frac{d\sigma}{\sqrt{2}}
\end{equation}	
where
 \begin{align*}
P(T):=\int_{D(T)}  \frac{t+c}{2} \left[   \left(  \frac{Xu}{t+c}   \right)^2 +  \left( |\grad u|^2 - ( \frac{x}{t+c}\cdot\grad u)^2 \right) \right] + \frac{u^2}{t+c}    +    \frac{t+c}{6} u^6 \, dx
\end{align*}

An application of 
H\"older shows that 

\begin{align*}
\int_{L(I)} (t+c)\left( \frac{Xu}{t+c}\right)^2 \,\frac{d\sigma}{\sqrt 2} 
	&\lesssim (T_2+c) \int_{L(I)} (\pat u+ \pa_r u)^2 \,d\sigma +\int \frac{u^2}{t+c} \, d\sigma \\
	&\lesssim (T_2+c) ( \| Lu \|_{L^2(L(I))}^2 + \| u\|_{L^6(L(I))}^2 )  
\end{align*}

In summary, 

\begin{align*}
T_2\int_{D(T_2)}u^6dx \lesssim P(T_2) +\iint_{\Gamma(I)} u^6 &\lesssim P(T_1) + (T_2+c) G(\flux(I)) \\
&\lesssim (T_1+c)E_{|x|<T_1+c}(T_1) + (T_2+c)G(\flux(I))
\end{align*}
	and $G(\theta):=\theta+\theta^{1/3}$ is a function that decays to zero as its argument decays to zero.  Take $T_1 = \delta T_2$ to see that, since $\delta$ was arbitrary and the flux decays,
	$$\limsup_{t \to \infty} \| u(t,\cdot) \|_{L^6(\R^3)} =0.$$

\section{$L^6$ norm decay and scattering of solutions on small asymptotically flat perturbations of Minkowski space}

We now come to the main result and its proof. 

\begin{theorem}[Main Theorem] \label{mainthm}
Let $u$ be the unique global strong solution of \eqref{cp3}. \begin{enumerate}
\item
We make the following assumptions on the perturbation $h$: 
\begin{equation}\label{deriv}
|\pa h |\lesssim \epsilon \im^{- 1 -\gamma}
\end{equation} 
\begin{equation}\label{goodh}
|h|\ls\eps \frac{\la t-r \ra^{1/2}}{\im^\gamma \la t+r\ra^{1/2}}
\end{equation}
\begin{equation}\label{badh}
|h^{\underline{L}\underline{L}} | \lesssim \eps\frac{\la t-r \ra}{\im^\gamma\la t+r\ra}
\end{equation}
 where $\gamma>0$ is an arbitrarily small constant and $\epsilon>0$ is a sufficiently small constant.
 
 Then \begin{equation}\label{potdecay}\limsup_{t\to\infty} \| u(t,\cdot) \|_{L^6(\R^3)} = 0.\end{equation}
\item
If in addition 
\begin{equation}\label{twoderiv}
|\pa^J h |\lesssim \epsilon \im^{-2- \gamma}, \quad |J|=2,
\end{equation}
then $u$ scatters in the energy space.
\end{enumerate}
\end{theorem}

Recall that we define the normal derivative to the cone
$$\underline{L} = \frac{x^i}{|x|}\pai - \pa_0 = \frac{x^i}{|x|}\pai - \pat$$
and
$$\hll=h^{\alpha\beta}\underline{L}_\alpha \underline{L}_\beta = h^{00} - 2 \sum_i h^{0i} \frac{x^i}{|x|} + \sum_{i, j} h^{ij} \frac{x^i x^j}{|x|^2}$$.

We also remark that the decay rates on $h$ and $\partial h$ are consistent with the ones required for local energy decay Theorem~\ref{LE} except near the cone $t\approx |x|$, where we need better decay rates to close the argument.

\begin{proof} 

Recall that we define $$G(\theta):=\theta+\theta^{1/3}.$$

The main estimate 
in this proof is the following: 

For any $R\geq 0$ and $1< T_1$ \footnote{We shall only be interested in certain sufficiently large values of $T_1$ and $R$.},  and any $T_1+R <T_2$,
we have \begin{equation} \label{mainest}
\int_{\R^3}u^6(T_2,x)dx \ls \frac{T_1+R+1}{T_2} E_{\{|x|<T_1+R+1\}} (T_1)  + \frac{E}{T_2^\gamma} 		+ G\Bigl(E_{\{|x|>T_1+R\}}(T_1) + \la R \ra \|u\|^2_{LE^1[T_1,T_2]}\Bigr)   
\end{equation} 

Here, and for the rest of the proof, the implicit constant does not depend on $T_1$, $R$, $T_2$, or $\epsilon$.

Let us finish the proof of \eqref{potdecay}, assuming that \eqref{mainest} holds. Pick any $\tilde \eps>0$, and let $T_1$ be large enough such that $$\|u\|^2_{LE^1[T_1,\infty)} < \tilde\eps;$$ such a number may be found because of the local energy estimate \eqref{LE}. Next pick $R$ large enough so that
\[
E_{\{|x|>T_1+R\}}(T_1) < \tilde \eps
\]

Now let $T_2\to\infty$ in \eqref{mainest}. We obtain
\[
\limsup_{T_2\to\infty} \int_{\R^3}u^6(T_2,x)dx \lesssim \tilde \eps + \tilde\eps^{1/3}
\]
and \eqref{potdecay} follows by letting $\tilde \eps\to 0$.

Let us first assume that $u$ is a classical solution to the equation
\begin{equation}\label{approx}
Pu = u^5 + F
\end{equation}

We will prove that
\begin{equation}\label{potdecay2}
\begin{split}
& \int_{\R^3}u^6(T_2,x)dx \ls \frac{T_1+R+1}{T_2} E_{\{|x|<T_1+R+1\}}(T_1)  + \frac{E}{T_2^\gamma} \\ &
	 + G\Bigl(E_{\{|x|>T_1+R\}}(T_1) + \la R \ra \|u\|^2_{LE^1[T_1,T_2]}  + \|F\|_{L^1[T_1, T_2]L^2}\|\nabla u\|_{L^{\infty}[T_1, T_2]L^2}\Bigr)
\end{split}\end{equation} where $E:=E(0)$.

A similar argument as the one in Section 4 allows us to pass to the limit and deduce \eqref{mainest} from \eqref{potdecay2}.

We first observe that by averaging we know that there is $c\in [R, R+1]$ so that
\begin{equation}\label{ave}
\int_{L_c([T_1,T_2])} \frac{|\nabla u|^2}{\im^{1+\gamma}}\, d\sigma \leq \|u\|^2_{LE^1[T_1,T_2]}.
\end{equation}

For the rest of this proof, fix $c$ as above. Note that the hypothesis $T_1+R < T_2$ implies that $T_2 \approx T_2+c$.

\subsection{$L^6$ norm decay of solutions on spacelike slices exterior to the cone}

The next lemma shows that we can control both the outside energy and the flux through $L_c$. Note that, unlike in the Minkowski case, it is not clear that this can be done for all $c$.

\begin{lemma} \label{6.4}
Let $u$ solve \eqref{approx}. Then 
\begin{equation}\begin{split}\label{extu6}
E_{\{|x|>T_2+c\}}(T_2) + \flux[T_1,T_2] \ls E_{\{|x|>T_1+c\}}(T_1) + \la R\ra \|u\|^2_{LE^1[T_1,T_2]} \\ + \|F\|_{L^1[T_1, T_2]L^2}\|\nabla u\|_{L^{\infty}[T_1, T_2]L^2}.
\end{split}\end{equation}

Here, 
\[ \flux[T_1,T_2]:=\int_{L_c([T_1,T_2])} \frac{1}2 |\bar\pa u|^2 + \frac{1}6u^6\,\frac{d\sigma}{\sqrt{2}} \]
and
$$\bar\pa u:=\{Lu, (r\sin\phi)^{-1}\pa_\theta, r^{-1}\pa_\phi\}, \quad L=  \frac{x^i}{|x|}\pai + \pat$$ denote the tangential derivatives of $u$ to the light cone.
\end{lemma}

\begin{proof}
Let $I = [T_1,T_2]$. Multiplying both sides of the equation in \eqref{approx} by $\pat u$, we obtain the identity $$\paa(\gab\pab u \pat u) - \frac{1}{2}\pat(\gab\pab u \paa u) + \frac{1}{2}\pat\gab\pab u \paa u = \frac{1}{6}\pat (u^6) + F\pat u.$$ Define 
\[\gext(I) := \{T_1 \leq t \leq T_2, |x|>t+c\}, \quad D(T)^c := \gext(I)\cap\{t=T\}\] 

Applying the divergence theorem within the region $\gext(I)$ leads to 
\begin{equation} \label{6.4.1} \begin{split} \iint_{\gext(I)} F\pat u \, dxdt  =  & \iint_{\gext(I)} \frac{1}{2}\pat\gab\pab u \paa u \,dxdt +\\ & \int_{\pa \gext(I)} \nua\gab\pab u \pat u - \frac{1}2\nu_0\gab\pab u \paa u - \frac{1}6 \nu_0 u^6\,d\sigma.\end{split} \end{equation}

Next, let $\bd$
denote $$\bd = \nua\hab\pab u \pat u - \frac{1}2\nu_0\hab\pab u \paa u;$$  clearly $\bd$ depends on the domain of integration. 
Expanding \eqref{6.4.1}, we have
\begin{equation} \label{6.4.2} \begin{split}
& E_{\{|x|>T_2+c\}}(T_2)+\int_{D(T_2)^c}\bd\,dx + \flux(I) + \iint_{\gext(I)} \frac{1}{2}\pat\hab\paa u \pab u \,dxdt\\ &= \iint_{\gext(I)} F\pat u \, dxdt + E_{\{|x|>T_1+c\}}(T_1) + \int_{D(T_1)^c}\bd \,dx  + \int_{L_c(I)}\bd\,d\sigma 	
\end{split}\end{equation} 

The space-time term is easy to estimate by \eqref{deriv} 
\begin{equation}\label{spacetime}
\iint_{\gext(I)} \frac{1}{2}\pat\hab\paa u \pab u \,dxdt\ls \iint_{\gext(I)} \frac{|\nabla u|^2}{\im^{1+\gamma}}\,dxdt\leq \|u\|^2_{LE^1[T_1, T_2]}.
\end{equation}

Similarly, using that $|h|\lesssim \epsilon$, we obtain
\begin{equation}\label{timeslice}
\int_{D(T_j)^c}\bd\,dx \lesssim \epsilon E_{\{|x|>T_2+c\}}(T_j), \quad j=1,2.
\end{equation} 
which can be absorbed in $E_{\{|x|>T_2+c\}}(T_j)$ for small enough $\epsilon$.
 
Finally, we need to estimate the perturbative error term on the lateral boundary; this is where we will use \eqref{goodh} and \eqref{badh}. We write 
\begin{align*} 
& \nu_\alpha \hab \pab = -\frac12 \hll \underline{L} + O(h)\bar\pa \\ 
& \pat = \frac12(L-\underline{L}) \\
&\hab \paa u \pab u= \frac14 \hll (\underline{L} u)^2 + O(h)\bar\pa u \pa u 
\end{align*} 

Note that, due to \eqref{badh} and \eqref{goodh} we have that on $L(I)$
\begin{equation}\label{hcone}
\hll \lesssim \frac{\la R\ra}{\la x\ra^{1+\gamma}} , \quad |h| \lesssim \frac{\epsilon\la R\ra^{1/2}}{\la x\ra^{1/2+\gamma}}
\end{equation}
and thus by Cauchy-Schwarz
\begin{equation}\label{lat}
\int_{L_c(I)}\bd\,d\sigma \lesssim \int_{L_c(I)} |\hll| (\underline{L} u)^2 + |h| |\bar\pa u| |\pa u| d\sigma \lesssim \epsilon \flux(I) + \la R\ra\int_{L_c(I)} \frac{|\nabla u|^2}{\im^{1+\gamma}}\, d\sigma 
\end{equation}

The conclusion of the lemma now follows from \eqref{ave}, \eqref{6.4.2}, \eqref{spacetime}, \eqref{timeslice}, and \eqref{lat}.
\end{proof}

\subsection{$L^6$ norm decay of solutions on interior spacelike slices of the cone}

The objective within this section is to show that solutions to \eqref{approx} satisfy
\begin{equation}\begin{split}\label{intu6}
\int_{D(T_2)}u^6(T_2,x)dx & \ls \frac{T_1+R+1}{T_2} E_{\{|x|<T_1+c\}}(T_1) + \frac{E}{T_2^{\gamma}} + G(\flux[T_1, T_2]) + \\& \hspace{1cm} \la R \ra \|u\|^2_{LE^1[T_1,T_2]}   + \|F\|_{L^1[T_1, T_2]L^2}\|\nabla u\|_{L^{\infty}[T_1, T_2]L^2}
\end{split}\end{equation}

The desired estimate \eqref{potdecay2} now follows from \eqref{extu6} and \eqref{intu6}.

To prove \eqref{intu6}, we multiply both sides of \eqref{approx} by $Xu$ where $X := S + c\pat + 1$ and $S$ is the scaling vector field  and obtain 
\begin{equation} \label{sc}
\begin{split}
\paa(\gab \pab u Xu) - \frac{1}{2} \pat((t+c) \gab \pab u \paa u) - \frac{1}{2} \pai(x^i \gab \pab u \paa u) \\ + \frac{1}{2} ((X-1) \gab) \paa u \pab u 
	 = \pat((t+c)\frac{u^6}6) + \pai(x^i\frac{u^6}6) -\frac{u^6}3 + F Xu.
\end{split}\end{equation} 	
	Indeed, by the symmetry of $\gab$, 
$$\paa(\gab \pab u \pat u) - \frac{1}{2} \pat(\gab \pab u \paa u) + \frac{1}{2} \pat \gab \paa u \pab u = (Pu) \pat u.$$
and similarly
\begin{align*}\paa(\gab \pab u t \pat u) - \frac{1}{2} \pat ( t \gab \pab u \paa u) - g^{0\beta} \pab u \pat u + \frac{1}{2} \gab \paa u \pab u \\+ \frac{1}{2} t \pat \gab \pab u \paa u = (Pu)t\pat u\end{align*}

\begin{align*}\paa(\gab \pab u x^j \pa_j u) - g^{j\beta} \pab u \paj u - \frac{1}{2} \paj(\gab \pab u x^j \paa u) + \frac{1}{2} x^j \paj \gab \pab u \paa u \\ + \frac{3}{2} \gab \pab u \paa u = (Pu)(x^j\pa_j u)\end{align*}
$$\paa(\gab u\pab u)-\gab\pab u \paa u=(Pu)u.$$
The nonlinear term follows in a similar manner. Upon summing these terms we obtain \eqref{sc}.

We now integrate \eqref{sc} on $\Gamma(I)$ and apply the divergence theorem. We obtain
\begin{align*}
\iint_{\Gamma(I)} \frac{u^6}3 + \frac{1}{2} ((X-1) \gab) \paa u \pab u - F Xu \, dxdt = \int_{\pa \Gamma(I)} \nu_\alpha \gab \pab u Xu - \\ \frac{1}{2} \nu \cdot (t+c,x) \gab \pab u \paa u - \frac{1}{2} \nu \cdot (t+c,x) \frac{u^6}6\, d\sigma
\end{align*}
	
Recall that on $L(I)$ the outward unit normal vector $\nu$ to $L(I)$ is $(-1,x/|x|)/\sqrt{2}$, and thus $\nu \cdot (t+c,x) = 0$ on $L(I)$. The boundary term can be now written more explicitly as \[
-P(T_2) + P(T_1) + \flux(I)	+ BDR_h 
	\]
where the first three terms come from the Minkowski case, and
\begin{align*}
BDR_h = & \int_{D(T_2)} h^{0\beta}\pab uXu - \frac{1}{2}(t+c) \hab \pab u \paa u \,dx \\ & - \int_{D(T_1)} h^{0\beta}\pab uXu  -\frac{1}{2}(t+c) \hab \pab u \paa u \,dx 
 + \int_{L(I)} \nu_\alpha \hab \pab u Xu\, dx 
\end{align*}	
	  
As explained in Section 5, we know that
\[
P(T_2) \gtrsim T_2 \int_{D(T_2)}u^6(T_2,x)dx
\]
\[
P(T_1) \lesssim (T_1+c) E_{\{|x|<T_1+c\}}(T_1)
\]
We can also make the trivial estimate
\[
\iint_{\Gamma(I)} F Xu \, dxdt \lesssim T_2 \|F\|_{L^1[T_1, T_2]L^2}\|\nabla u\|_{L^{\infty}[T_1, T_2]L^2}
\]	 
Moreover, our assumptions on $\hab$ immediately imply that 
\[
(X-1)\hab \lesssim t\im^{-1-\gamma}
\] 
and thus
\begin{equation}
\iint_{\Gamma(I)} | (X-1)\gab \paa u \pab u | \,dxdt \lesssim T_2 \iint_{\Gamma(I)} \frac{|\nabla u|^2}{\im^{1+\gamma}}\,dxdt \leq T_2 \|u\|^2_{LE^1[T_1,T_2]} 
\end{equation} 

The conclusion \eqref{intu6} will follow if we show that	
\[
BDR_h \lesssim \epsilon (P(T_2)+P(T_1)) + T_2^{1-\gamma} E + T_2 \left( G(\flux(I)) + \la R\ra \LE \right)
\]	

\

Let us write
\[
D(T_2) = D_{int}(T_2) \cup D_{ext}(T_2)
\]
where
\[
D_{int}(T_2) = D(T_2)\cap\{|x|\leq \frac{T_2+c}2\}, \quad D_{ext}(T_2) = D(T_2)\cap\{|x|\geq \frac{T_2+c}2\},
\]

Since $|h| \lesssim \epsilon$ in $D_{int}$, we have that
\[
\int_{D_{int}(T_2)} h^{0\beta}\pab uXu - \frac{1}{2}(t+c) \hab \pab u \paa u \,dx \lesssim \epsilon P(T_2)
\]

On the other hand, $|h|\lesssim \frac{1}{T_2^{\gamma}}$ in $D_{ext}$, and thus by the boundedness of energy
\[
\int_{D_{ext}(T_2)} h^{0\beta}\pab uXu - \frac{1}{2}(t+c) \hab \pab u \paa u \,dx \lesssim T_2^{1-\gamma} E(T_2) \lesssim T_2^{1-\gamma} E
\]

Adding the last two inequalities we obtain
\[
\int_{D(T_2)} h^{0\beta}\pab uXu - \frac{1}{2}(t+c) \hab \pab u \paa u \,dx \lesssim \epsilon P(T_2) + T_2^{1-\gamma} E
\]

Similarly we can show that
\[
\int_{D(T_1)} h^{0\beta}\pab uXu - \frac{1}{2}(t+c) \hab \pab u \paa u \,dx \lesssim \epsilon P(T_1) + T_1^{1-\gamma} E
\]

\

We are left with dealing with the lateral terms. We will show that
\begin{equation}\label{latS}
\int_{L(I)} \nu_\alpha \hab \pab u Xu\, d\sigma \lesssim T_2 \left( G(\flux(I)) + \la R\ra \LE \right)
\end{equation}

We first remark that $Xu = (rL + 1)u$ on $L(I)$, and we again write
\begin{equation}
\nu_\alpha \hab \pab u  = -\frac12 \hll \underline{L}u + O(h)\bar\pa u
\end{equation} 

Note that \eqref{hcone} in particular imply the weaker estimates
\[
\hll \lesssim \frac{\la R\ra^{1/2}}{\la x\ra^{1/2+\gamma}}, \quad h \lesssim 1
\]

We can now estimate by Cauchy-Schwarz, \eqref{hcone} and the fact that $r\leq T_2 + c \lesssim T_2$:

\begin{align*}
\int_{L(I)} |\hll \underline L u (rLu)|\,d\sigma &\lesssim T_2\int_{L(I)} |\frac{\la R\ra^{1/2}}{\im^{1/2+\gamma}}\underline Lu Lu|\,dx \\
	&\leq T_2 \left(\la R\ra \int_{L(I)} \frac{|\nabla u|^2}{\im^{1+\gamma}} \,dx + \int_{L(I)} |Lu|^2\,dx \right) \\
	&\lesssim T_2 (\la R\ra \|u\|^2_{LE^1[T_1,T_2]} + \flux[T_1, T_2])
	\end{align*} 
where in the last inequality we used \eqref{ave}.
		
Similarly, 
\begin{align*}
\int_{L(I)} |h \bar\pa u (rLu)|\,dx &\lesssim T_2 \int_{L(I)} |\bar\pa u Lu|\, dx \\
&\leq T_2\int_{L(I)} |\bar\pa u|^2\,dx \lesssim T_2 \flux[T_1, T_2]
\end{align*} 
	
Thirdly, by an application of \eqref{hcone}, Cauchy-Schwarz and then H\"older's inequality,
\begin{align*}
\int_{L(I)} |\hll \underline Lu u|\,dx 
	&\ls \int_{L(I)} \frac{\la R \ra^{1/2}}{\im^{1/2+\gamma}} |\underline L u| |u| \,dx  \\
	&\ls \left(\int_{L(I)} \la R \ra \frac{\ed}{\im^{1+\gamma}}\right)^{1/2} \left(\int_{L(I)} u^2 \,dx\right)^{1/2} \\
	&\ls  	\left(\int_{L(I)} \la R \ra \frac{\ed}{\im^{1+\gamma}}\right)^{1/2}  \|u\|_{L^6( L(I))} T_2 \\
	&\ls T_2\left( \la R \ra \LE + \flux(I)^{1/3} \right)
\end{align*} 
where again we used \eqref{ave}.	
	
Finally, 
	\begin{align*}
	\int_{L(I)} | h \bar\pa u u|dx 
		&\ls (\int_{L(I)} |\bar\pa u|^2)^{1/2} (\int_{L(I)} u^2dx)^{1/2} \\
		&\ls (\int_{L(I)} |\bar\pa u|^2)^{1/2} (\int_{L(I)} |u|^6 dx)^{1/6} T_2 \\
		&\ls T_2 G(\flux[T_1, T_2])
	\end{align*} 

The last four estimates now imply \eqref{latS}, which finishes the proof of \eqref{intu6}.

\subsection{Scattering}
To obtain part (2) of the theorem, note that if $\pa^J h \lesssim \eps \im^{-|J|-\gamma}$ where $|J| \leq 2$ (which are implied by our assumptions in the main theorem), then global Strichartz estimates are implied by a refinement of the local energy decay estimates (see Theorem 6 in \cite{mt2}). Then, for any $\eta>0$, by choosing a sufficiently large number $T>0$, we obtain $\| w \|_{L^5L^{10}([ T,\infty) \times \R^3)} \leq \eta$ where $w$ solves \eqref{main}. For any $\tilde w$ with $\| \tilde w \|_{L^5L^{10}([ T,\infty) \times \R^3)} \leq \eta$, let $W$ be the solution to $$P W = (w + \tilde w)^5$$ with $$\lim_{t\to\infty} \|\nabla W(t,\cdot)\|_{L^2(\R^3)} =0.$$ As $\eta>0$ was arbitrary, we may select $\eta$ sufficiently small so that the map $\tilde w \mapsto W$ is a contraction mapping, so that for any finite energy solution $w$ of \eqref{main}, there exists a unique solution to \eqref{linear} such that their difference vanishes in the $\dot H^1 \times L^2$ norm as $t \to \infty$, and we conclude that the solution scatters in the energy space.

\end{proof}

\end{document}